\documentclass[review]{elsarticle}
\usepackage{amsthm}
\usepackage{lineno,hyperref}
\usepackage{amssymb,enumerate}
\usepackage{amsmath}
\usepackage{graphics}
\usepackage{xcolor}

\usepackage{geometry}
\newtheorem{thm}{Theorem}[section]
\newtheorem{theorem}{Theorem}[section]

  \newtheorem{lemma}[thm]{Lemma}
 
  \newtheorem{proposition}[thm]{Proposition}
 
 \newtheorem{definition}[thm]{Definition}
 
  \newtheorem{remark}[thm]{Remark}

\newcommand{\ep}{\varepsilon}

    \DeclareMathOperator\supp{supp}

\journal{Journal of \LaTeX\ Templates}

\begin{document}

\begin{frontmatter}

\title{Small data blow-up of semi-linear wave equation with scattering dissipation and time-dependent mass}

\author{Masahiro Ikeda
\footnote{Department of Mathematics, Faculty of Science and Technology, Keio University, 3-14-1 Hiyoshi, Kohoku-ku, Yokohama, 223-8522, Japan/Center for Advanced Intelligence Project, RIKEN, Japan. e-mail: masahiro.ikeda@keio.jp/masahiro.ikeda@riken.jp}
\quad
Ziheng Tu
\footnote{Department of Mathematics, School of Data Science, Zhejiang University of Finance and Economics, 310018, Hangzhou, P.R.China. e-mail: tuziheng@zufe.edu.cn}
\quad
Kyouhei Wakasa
\footnote{Department of Creative Engineering, National Institute of Technology, Kushiro College, 2-32-1 Otanoshike-Nishi, Kushiro-Shi, Hokkaido 084-0916, Japan. e-mail: wakasa@kushiro-ct.ac.jp}
}

\begin{abstract}
In the present paper, we study small data blow-up of the semi-linear wave equation with a scattering dissipation term and a time-dependent mass term from the aspect of wave-like behavior. The Strauss type critical exponent is determined and blow-up results are obtained to both sub-critical and critical cases with corresponding upper bound lifespan estimates. For the sub-critical case, our argument does not rely on the sign condition of dissipation and mass, which gives the extension of the result in \cite{Lai-Sch-Taka18}. Moreover, we show the blow-up result for the critical case which is a new result.
\end{abstract}


\end{frontmatter}

\section{Introduction and main result}

In the present paper, we consider blow-up problem of the following hyperbolic model
\begin{eqnarray}\label{main}
u_{tt}-\Delta_g u+a(t)u_t+b(t)u=|u|^p\ &(x,t)\ \in\ \mathbb{R}^n\times[0,\infty),\nonumber\\
u(x,0)=\varepsilon f(x)\ &x\ \in\ \mathbb{R}^n,\\
u_t(x,0)=\varepsilon g(x)\ &x\ \in\ \mathbb{R}^n,\nonumber
\end{eqnarray}
where $a(t)$ and $b(t)$ stand for the coefficients of damping and potential terms such that $a(t)\in C^1([0,\infty))$, $b(t)\in C([0,\infty))$ and $a(t),\ tb(t)\in L^1([0,\infty))$.
Throughout this paper, we assume that $\varepsilon>0$ is a small parameter and the initial data $f\in\ H^1(\mathbb{R}^n)$ and $g\in\ L^2(\mathbb{R}^n)$.
We assume that the perturbations of Laplacian are uniformly elliptic operators
$$\Delta_g =\sum_{i,j=1}^n\partial_{x_i}(g_{ij}(x)\partial_{x_j})$$
whose coefficients satisfy, with some $\alpha>0,$ and $\gamma>0$
the following:
\begin{equation}
\label{uniform-g}
\sum_{i,j=1}^n g_{ij}(x)\xi_i\xi_j  \geq \gamma |\xi|^2,\quad         \xi\in\mathbb{R}^n,\\
\end{equation}
\begin{equation}
\label{g}
g_{ij}\in C^1(\mathbb{R}^n),\quad |\nabla g_{ij}(x)|+|g_{ij}(x)-\delta_{ij}|=O(e^{-\alpha|x|})  \mbox{ as } |x|\rightarrow\infty,
\end{equation}
where $\delta_{ij}$ is the Kronecker's delta.
We briefly review several previous results concerning (\ref{main}) with various types of setting on $a$ and $b\equiv0$ when $g_{ij}=\delta_{ij}$.
When $a(t)=1$, Todorova and Yordanov \cite{TY} showed the blow-up result
if $1<p<p_F(n)$, where $p_F(n)=1+2/n$ is the Fujita exponent known
to be the critical exponent for the semilinear heat equation.
The same work also obtained small data global existence for $p>p_F(n)$.
Zhang \cite{Z} established the blow-up in the critical case $p=p_F(n)$ even if the data are small.

For the non-constant damping case, there are extensive discussion on this topic when $a(t)$ is of the form $\mu/(1+t)^{\beta}$, with $\mu>0$ and $\beta\in\mathbb{R}$. As the classification given in Wirth \cite{Wir04,Wir06,Wir07}, for the ``effective case'', $-1<\beta<1$, Lin, Nishihara and Zhai \cite{LNZ12} obtained a small data blow-up result, if $1<p\le p_F(n)$, and small data global existence result, if $p>p_F(n)$; see also D'Abbicco, Lucente and Reissig~\cite{DLR13}.
When $\beta=-1$, Wakasugi \cite{W17} obtained a small data global existence for exponents
$p_F(n)<p<n/[n-2]_+$, where
\[
[n-2]_{+}:=
\left\{
\begin{array}{lll}
\infty &\mbox{for}& n=1,2,\\
n/(n-2) &\mbox{for}& n\ge3.
\end{array}
\right.
\]
On the other hand, Fujiwara, Ikeda and Wakasugi \cite{FIW19} have obtained the blow-up results together with the sharp estimates of the lifespan for $p<p_F(n)$.
When $p=p_F(n)$, Fujiwara, Ikeda and Wakasugi \cite{FIW19} proved sharp lower estimate of the lifespan and Ikeda and Inui \cite{II19} give the sharp upper estimate of the lifespan (see also \cite{ISW} for further discussion).

The case of $\beta=1$ is the threshold between ``effective case'' and ``scattering case''. It is known that the solution of corresponding linear problem shows different asymptotic behaviors for different size of $\mu$. While for the blow-up problem for semi-linear equation, we expect $\mu$ also plays the important role on determining the critical exponent.
Wakasugi~\cite{WY14_scale} showed the blow-up, if $1<p\leq p_F(n)$ and $\mu>1$ or $1<p\leq 1+2/(n+\mu-1)$ and $0<\mu\leq 1$.
Moreover, D'Abbicco~\cite{DABI} verified the global existence, if $p>p_F(n)$ and
$\mu$ satisfies one of the following: $\mu\geq5/3$ for $n=1$, $\mu\geq3$ for $n=2$ and
$\mu\ge n+2$ for $n\ge3$. An interesting observation is that the Liouville substitution $w(x,t):=(1+t)^{\mu/2}u(x,t)$
transforms the damped wave equation (\ref{main}) into the Klein-Gordon type equation
\[
w_{tt}-\Delta w+\frac{\mu(2-\mu)}{4(1+t)^2}w=\frac{|w|^p}{(1+t)^{\mu(p-1)/2}}.
\]
Thus, one expects that the critical exponent for $\mu=2$ is related to that of the semilinear wave equation.
D'Abbicco, Lucente and Reissig \cite{DLR14} have actually obtained
the corresponding blow-up result, if $1<p<p_c(n):=\max\left\{p_{F}(n),\ p_S(n+2)\right\}$. Here
$p_S(n)$ is the so-called Strauss exponent, which is critical exponent for the semilinear wave equation,
\begin{equation}
\label{strauss}
p_S(n):=\frac{n+1+\sqrt{n^2+10n-7}}{2(n-1)}
\end{equation}
which is the positive root of the quadratic equation
\begin{equation}
\label{gamma}
\gamma(p,n):=2+(n+1)p-(n-1)p^2=0.
\end{equation}
Their work also showed the existence of global classical solutions for small $\ep>0$, if $p>p_c(n)$ and either $n=2$ or $n=3$ and the data are radially symmetric.
Lai, Takamura and Wakasa \cite{LTW} have obtained the blow-up part of Strauss' conjecture, together with an upper bound of the lifespan $T(\ep)$, for (\ref{main}) in the case $n\geq 2$, $0<\mu<(n^2+n+2)/2(n+2)$ and $p_F(n)\le p<p_S(n+2\mu)$.
Later, Ikeda and Sobajima \cite{IS} were able to replace these conditions by less restrictive
$0<\mu<(n^2+n+2)/(n+2)$ and $p_F(n)\leq p\leq p_S(n+\mu)$ and proved small data blow-up result in the case $n=1$ and $0<\mu<\frac{4}{3}$. In addition, they have derived an
upper bound on the lifespan in the case $n\ge 1$ and $0<\mu<(n^2+n+2)/(n+2)$. Tu and Lin \cite{TL1}, \cite{TL2} have improved the estimates of $T(\ep)$ in \cite{IS} in the case $n\ge 2$ and $p_F(n)\le p< p_S(n+\mu)$. Unfortunately, still there is a gap for $\mu$, we have no results. In the case of $\beta=-1,\ p=p_F(n)$ and $\mu>1$, a small data blow-up result and an upper bound of lifespan are obtained by Ikeda, Sobajima and Wakasugi \cite{ISW}.

For the overdamping case $\beta< -1$, the long time behavior of solutions to (\ref{main}) is quite different. In the case of $\beta<-1$, Ikeda and Wakasugi \cite{IW} proved a small data global existence for energy solution for any $1<p\le 1+4/(n-2)$.

Lastly, for the scattering case $\beta>1$, we expect the critical exponent to be exactly the Strauss exponent.
In fact, Lai and Takamura \cite{LT} have shown that the solutions of (\ref{main}) blows up in finite time when $1<p<p_S(n)$ with the lifespan upper bound $T(\ep) \le C\ep^{-2p(p-1)/\gamma(p,n)}$ for $n\ge2$ and $1<p<p_S(n)$. In the critical case, $p=p_S(n)$, Wakasa and Yordanov \cite{WY-damped} showed the blow-up result together with the lifespan $T(\ep)\le \exp(C\ep^{-p(p-1)})$ for
$p=p_S(n)$ and $n\ge2$. These results of lifespan is sharp, since same type estimates were given for the semi-linear wave equation by Takamura \cite{Ta15}. On the other hand, Liu and Wang \cite{LW18} have obtained global existence results with super-critical exponent $p>p_S(n)$ for $n=3,4$ on asymptotically Euclidean manifolds.

We now turn back to our original model (\ref{main}). So far as we know, the Cauchy problem of this model has been studied in \cite{Lai-Sch-Taka18}, \cite{Lai-Sch-Taka} and \cite{DGR19} with different types of assumption on the time dependent damping and mass terms. In \cite{DGR19}, D'Abbicco, Girardi and Reissig studied the interplay between the effective damping and the dominated mass, i.e, $b(t)=o(a(t))$ as $t\rightarrow \infty$. The decay estimate of linear solution was obtained and has been successfully applied to the power type nonlinearity equation so that to show the global in time solution. Moreover, the scale of critical exponent $1+\frac{4}{n+4\beta}$ has been found which depends on $\beta:=\liminf_{t\rightarrow\infty}b^2(t)\int_0^t1/a(\tau)d\tau$. For the scattering damping case, Lai, Schiavone and Takamura has studied the blow up problems but with mass terms $b(t)=-\mu_2/(1+t)^{\alpha+1}$ of different decay speed in \cite{Lai-Sch-Taka18} (for $\alpha>1$) and \cite{Lai-Sch-Taka} (for $0<\alpha<1$), respectively. By applying multiplier method with comparison arguments, a wave-like blow-up result has been obtained for fast decaying mass in \cite{Lai-Sch-Taka18}. On the contrary, when the decay of mass term is not so fast, they find very different blow up behavior. The lifespan upper bound is much shorter. It seems in this case, the mass term plays dominant role.
In the present paper, we aim to extend the result of \cite{Lai-Sch-Taka18}. In their proof, the coefficient of mass term $b(t)$ is required to be negative. We aim to  remove this technique requirement. Moreover, we also obtain the blow-up result for the critical case $p=p_S(n)$, for the nonnegative mass term. We call our treatment as ``double multiplier method'' as we apply the multiplier method twice on the equation. The multiplier method comes from \cite{PT}, where the semi-linear wave equation with the scale invariant damping and mass term is studied.

We first give the definition of energy solution of problem \eqref{main}.
\begin{definition}
We say that $u$ is an energy solution of \eqref{main} on $[0,T)$ if
$$u\in C([0,T),H^1(\mathbb{R}^n))\cap C^1([0,T),L^2(\mathbb{R}^n))\cap L^p_{loc}(\mathbb{R}^n\times[0,T))$$
and
\begin{equation}
\begin{aligned}
&\int_{\mathbb{R}^n}u_t(x,t)\phi(x,t)dx-\int_{\mathbb{R}^n}u_t(x,0)\phi(x,0)dx\\
&+\int_0^tds\int_{\mathbb{R}^n}\{-u_t(x,s)\phi_t(x,s)+\sum_{i,j=1}^ng_{i,j}(x)\partial_{x_i}u(x,s)\partial_{x_j}\phi(x,s)\}dx\label{def}\\
&+\int_0^tds\int_{\mathbb{R}^n}a(s)u_t(x,s)\phi(x,s)dx+\int_0^tds\int_{\mathbb{R}^n}b(s)u(x,s)\phi(x,s)dx\\
&=\int_0^tds\int_{\mathbb{R}^n}|u(x,s)|^p\phi(x,s)dx
\end{aligned}
\end{equation}
with any $\phi\in C_0^\infty(\mathbb{R}^n\times[0,T))$ and any $t\in [0,T)$.
\end{definition}
\begin{definition}
We define the upperbound of lifespan to the problem \eqref{main} as
$$T(\ep):=\sup\{T\in[0,\infty)|\mbox{there exists an energy solution}\ u\ \mbox{in}\ [0,T) \}.$$
\end{definition}

Our main results are stated in the following.
\begin{theorem}[Sub-critical case]\label{thm:sub}
Let $n\geq 2$ and $1< p <p_S(n)$. Let $a(t)$, $tb(t)\in L^1([0,\infty))$ and $ a(t)\in C^1([0,\infty))$, $b(t)\in C([0,\infty))$.
Assume that the non-negative initial data $f$ and $g$ satisfy following conditions
\begin{eqnarray}
&g(x)+r_2(0)f(x)\geq 0\label{ini}\\
&g(x)+(a(0)-\rho'(0))f(x)\geq 0\label{ini2}
\end{eqnarray}
with compact support.
Here $r_2(0)$ and $\rho'(0)$ are the initial value of the solution of \eqref{r_2} and the derivative of solution of \eqref{rho-eqn}. Suppose that a solution $u$ of \eqref{main} satisfies:
$$supp\ u\subset\{(x,t)\in \mathbb{R}^n\times [0,T): |x|\leq t+R\}.$$
Then, there exists a constant $\varepsilon_0=\varepsilon_0(f,g,n,p,a,b,R)$ such that the lifespan $T(\varepsilon)$ has to satisfy
$$T(\varepsilon)\leq C\varepsilon ^{-2p(p-1)/\gamma(p,n)}$$
for $0<\varepsilon\leq\varepsilon_0$, where $C$ is a positive constant independent of $\varepsilon$.
\end{theorem}
\begin{remark}
The blow-up result in the subcritical case agrees with the statement of Remark 1 given in \cite{LW18}. There is no sign requirement of $a(t)$ and $b(t)$ stated in our argument. It extends the blow-up results stated in \cite{Lai-Sch-Taka18} where only negative mass potential $b(t)<0$ is considered.
\end{remark}
\begin{remark}
For $n=2$ and $1<p\leq2$, it is known that the lifespan estimates can be improved as
$$T(\varepsilon)\leq \left\{
\begin{array}{lll}  C\varepsilon^{-(p-1)/(3-p)}\ \ \ &\mbox{for}\ \ 1<p<2\\
Ca(\varepsilon)\ \ \ &\mbox{for}\ \ p=2\end{array}
\right.
$$
for the wave equation and also for the wave equation with the positive scattering damping and the negative fast decay mass term when $g$ does not vanish identically. Here $a=a(\varepsilon)$ is a number satisfying
$$a^2\varepsilon^2\log(1+a)=1.$$ For example, see the introductions in \cite{Ta15} and Theorem $2$ and Theorem $4$ in \cite{Lai-Sch-Taka18}.
However, we do not know whether this lifespan estimate still holds in our settings without the sign condition.
\end{remark}

\begin{theorem}[Critical case]\label{thm:cri}
Let $n\geq 2$ and $p = p_S(n)$. Let $a(t), b(t)> 0$ and $a(t),\ tb(t) \in L^1([0,\infty))$ and $a(t)\in C^1([0,\infty))$, $b(t)\in C([0,\infty))$. Assume that $f,\ g$ satisfy \eqref{ini} with compact support. Suppose that a solution $u$ of \eqref{main} satisfies:
$$supp\ u\subset\{(x,t)\in \mathbb{R}^n\times [0,T): |x|\leq t+R\}.$$
Then, there exists a constant $\varepsilon_0=\varepsilon_0(f,g,n,p,a,b,R)$ such that the lifespan $T(\varepsilon)$ has to satisfy
$$T(\varepsilon)\leq \exp(C\varepsilon^{-p(p-1)})$$
for $0<\varepsilon\leq\varepsilon_0$, where $C$ is a positive constant independent of $\varepsilon$.
\end{theorem}
\begin{remark}
 The critical part results can be viewed as an extension of scattering damping wave equation without mass term studied in \cite{WY-damped}. However, here we still can not relax the non-negative condition of $b(t)$ due to some technique reason. See Remark \ref{rem-App} for more details.
\end{remark}
\begin{remark}
The proofs of local existence and finite propagation speeds
are given in Appendix which is mainly based on the argument
in \cite{Lai-Sch-Taka} and \cite{Ev1}.
\end{remark}

The rest of the paper is arranged as follows. In Section \ref{prelim}, we do some preliminary work for both cases. The double multiplier method is applied to obtain the lower bound of functional which gives the frame of the iteration argument. In Section \ref{pf:sub}, the suitable test function for the sub-critical case is derived with some basic properties from the conjugated equation. The iteration argument of Theorem \ref{thm:sub} follows similar as the scattering damping wave equation without mass term in \cite{LT}. In Section \ref{pf:cri}, we first collect some basic proof elements for critical case. Some auxiliary functions are introduced for the lower bound estimation and finally Theorem \ref{thm:cri} can be proved with some similar argument as \cite{WY}.

\section{Preliminaries}\label{prelim}

Let $u$ be an energy solution of \eqref{main} on $[0,T)$, consider the lower bound of following functional
$$G(t):=\int_{\mathbb{R}^n}u(x,t)dx.$$
 Choosing the test function $\phi=\phi(x,s)$ in \eqref{def} to satisfy $\phi\equiv1$ in $\{(x,s)\in \mathbb{R}^n \times [0,t]:|x|\leq s+R\}$, we obtain
\begin{eqnarray*}
&&\int_{\mathbb{R}^n}u_t(x,t)dx-\int_{\mathbb{R}^n}u_t(x,0)dx+\int_0^tds\int_{\mathbb{R}^n}[a(s)u_t(x,s)+b(s)u(x,s)]dx\\
&&=\int_0^tds\int_{\mathbb{R}^n}|u(x,s)|^pdx,
\end{eqnarray*}
i.e,
$$G'(t)-G'(0)+\int_0^t a(s)G'(s)+b(s)G(s)ds=\int_0^tds\int_{\mathbb{R}^n}|u(x,s)|^pdx.$$
Since all the quantities in this equation except $G'(t)$ is differentiable in $t$, so that so is $G'(t)$. Hence, we have
\begin{equation}\label{iden:G(t)}
G''(t)+a(t)G'(t)+b(t)G(t)=\int_{\mathbb{R}^n}|u(x,t)|^pdx.
\end{equation}
We now introduce the double multiplier method. Assume that the left-hand-side of the above identity can be written in the form
\begin{eqnarray*}&&G''(s)+a(s)G'(s)+b(s)G(s)\\
&&=[G'(s)+r_2(s)G(s)]'+r_1(s)[G'(s)+r_2(s)G(s)],
\end{eqnarray*}
where $r_1(t)$ and $r_2(t)$ are the two unknown functions satisfying
\begin{equation}
\begin{cases}
\begin{aligned}\label{r_1r_2}
&r_1(t)+r_2(t)=a(t),\\
&r_2'(t)+r_1(t)r_2(t)=b(t).
\end{aligned}
\end{cases}
\end{equation}
Multiplying $\exp(\int_K^sr_1(\tau)d\tau)$ on the both sides of equation \eqref{iden:G(t)} and integrating over $[0,t]$, we have:
$$\exp(\int_K^sr_1(\tau)d\tau)[G'(s)+r_2(s)G(s)]\Bigg|_0^t=\int_0^t\exp(\int_K^sr_1(\tau)d\tau)\int_{\mathbb{R}^n}|u|^pdxds.$$
Here $K$ is an any real number. As initial data assumption \eqref{ini} implies
\begin{equation}\label{G>0}
G'(0)+r_2(0)G(0)\geq0,
\end{equation}
we have
$$G'(t)+r_2(t)G(t)\geq\int_0^t\exp(\int_t^sr_1(\tau)d\tau)\int_{\mathbb{R}^n}|u|^pdxds.$$
Again multiplying $\exp(\int_K^sr_2(\tau)d\tau)$ and using $G(0)\geq0$, we conclude that for $0\leq s_1 \leq  s_2 \leq t$:
$$G(t)\geq\int_0^t e^{\int_t^{s_2}r_2(\tau)d\tau}\int_0^{s_2}e^{\int_{s_2}^{s_1}r_1(\tau)d\tau}\int_{\mathbb{R}^n}|u|^pdxds_1ds_2.$$
We now state following lemma about $r_1(t)$ and $r_2(t)$.
\begin{lemma}\label{lem:L^1}
If $a(t)$ and $tb(t)$ are in $L^1([0,\infty))$, then there exits a pair of solution $(r_1(t),\ r_2(t))$ in $L^1([0,\infty))\times L^1([0,\infty))$ to the ODE system \eqref{r_1r_2}.
Moreover, if $b(t)$ is positive, then $r_2(t)$ is negative for large $t$.
\end{lemma}
\begin{proof}
We postpone its proof in Appendix.
\end{proof}
In fact, plugging $r_1(t)=a(t)-r_2(t)$ into the second equation of \eqref{r_1r_2} gives the following Ricatti's equation for $r_2$
\begin{equation}\label{r_2}
r_2'(t)+a(t)r_2(t)-r_2^2(t)=b(t).
\end{equation} Since $a(t)\in L^1$, we only need to show $r_2(t)\in L^1$.

With the aid of this lemma, the lower bound of $G(t)$ can be further simplified:
\begin{equation}\label{low:G(t)}
G(t)\geq C_{r_1,r_2}\int_0^t\int_0^{s_2}\int_{\mathbb{R}^n}|u(s_1,x)|^pdxds_1ds_2.
\end{equation}
where $C_{r_1,r_2}=e^{-\|r_1\|_{L^1}-\|r_2\|_{L^1}}$. We may also define the quantities:
$$A_1(t,s)=\int_s^tr_1(\tau)d\tau,\ A_2(t,s)=\int_s^tr_2(\tau)d\tau\ \ \mbox{and}\ \ A(t,s)=A_1(t,s)+A_2(t,s).$$

Before going further, we also state following lemma related to our perturbed Laplacian.
\begin{lemma}
\label{lem:varphi}
Consider the following elliptic problem:
\begin{equation}
\label{eigen-pr}
\Delta_g \varphi_\lambda = \lambda^2\varphi_{\lambda},\quad x\in \mathbb{R}^n,
\end{equation}
where $\lambda\in (0,\alpha/2].$
Let $n\geq 2$. There exists a solution $\varphi_\lambda\in C^\infty(\mathbb{R}^n)$ to
(\ref{eigen-pr}), such that
\begin{equation}
\label{lem1:ineq}
|\varphi_\lambda(x)-\varphi(\lambda x)|\leq C_\alpha\lambda^{\theta},\quad x\in \mathbb{R}^n,
\quad \lambda\in (0,\alpha/2],
\end{equation}
where $\theta\in (0,1]$ and
$\varphi(x)=\int_{\mathbb{S}^{n-1}}e^{x\cdot\omega} dS_\omega\sim c_n|x|^{-(n-1)/2}e^{|x|},$
$c_n>0,$ as $|x|\rightarrow\infty.$
Moreover, $\varphi_\lambda(\: \cdot\:)-\varphi(\lambda \: \cdot)$ is a continuous $L^\infty (\mathbb{R}^n)$ valued
function of $\lambda\in (0,\alpha/2]$ and there exist positive constants $D_0,$ $D_1$ and $\lambda_0$, such that
\begin{equation}
\label{2sided}
D_0 \langle \lambda|x|\rangle^{-(n-1)/2}e^{\lambda|x|}\leq \varphi_\lambda(x)\leq D_1
\langle \lambda|x|\rangle^{-(n-1)/2}e^{\lambda|x|}, \quad x\in\mathbb{R}^n,
\end{equation}
holds whenever $0<\lambda\leq \lambda_0$. Here $ \langle \cdot \rangle=\sqrt{1+|\cdot|^2}$.
\end{lemma}
\begin{proof}
See Lemma 2.2 in \cite{WY}.
\end{proof}

\section{Sub-critical case: Proof of Theorem \ref{thm:sub}}\label{pf:sub}

In this section, we focus on the sub-critical case. Following lemma gives the asymptotic for the test function.
\begin{lemma}\label{lem:rho}Let $a(t),tb(t)\in L^1([0,\infty))$ and $a(t)\in C^1([0,\infty))$, $b(t)\in C([0,\infty))$. Consider the auxiliary second order ODE with positive parameter $\lambda$,
\begin{equation}\label{rho-eqn}
\left\{
\begin{array}{ll}
\rho''(t)-a\rho'(t)+(b-\lambda^2-a')\rho(t)=0,\\
\rho(0)=1,\ \rho(\infty)=0.
\end{array}
\right.
\end{equation}
There exists a solution $\rho(t)$ which decays as $e^{-\lambda t}$ for large $t$.
\end{lemma}
\begin{proof}
We postpone its proof in Appendix.
\end{proof}
With this lemma and lemma \ref{lem:varphi}, we can define the test function
$$\psi(t,x)=\rho(t)\varphi_{\lambda_0}(x)$$
for the sub-critical case to give the lower bound of solution's $L^p$ norm. The proof follows the same approach as Yordanov and Zhang \cite{YZ}.
\begin{remark}In the following, for notation's simplicity, we mix use $\varphi_{\lambda_0}$ as $\varphi$ without making confusion.
\end{remark}
\begin{lemma}\label{lem:u^p}
Assume that $a(t),\ tb(t)\in L^1([0,\infty))$ and $a(t)\in C^1([0,\infty))$, $b(t)\in C([0,\infty))$. Then,
there exists a positive constant $C_0 = C_0 (f,g,a,b,n,p,R,r_1,r_2,\lambda_0)$ such that
\begin{equation}\label{low:u^p}
\int_{\mathbb{R}^n}|u(x,t)|^pdx\geq C_0\varepsilon^p\langle t\rangle^{n-1-(n-1)p/2 }
\end{equation}
holds for $t\geq 0$.
\end{lemma}
\begin{proof}
Define the functional
$$G_1(t):=\int_{\mathbb{R}^n}u(x,t)\psi(x,t)dx,$$ then by H\"{o}lder inequality, we have
\begin{equation}\int_{\mathbb{R}^n}|u(x,t)|^pdx\geq\frac{|G_1(t)|^p}{(\int_{|x|\leq t+R}|\psi|^{p'}(t,x)dx)^{p-1}}.\label{factor}\end{equation}
We estimate the lower bound of $|G_1(t)|$ and upper bound of $\int_{|x|\leq t+R}|\psi|^{p'}(t,x)dx$ respectively.
From the definition of energy solution, we have
\begin{eqnarray*}
&&\int_0^t\int_{\mathbb{R}^n}u_{tt}\psi dxds-\int_{0}^t\int_{\mathbb{R}^n}u\Delta_g\psi dxds\\
&&+\int_0^t\int_{\mathbb{R}^n}\partial_s(a(s)\psi u)-\partial_s(a(s)\psi)u+b(s)\psi udxds=\int_0^t\int_{\mathbb{R}^n}|u|^p\psi dxds.
\end{eqnarray*}
Utilizing $\Delta_g\varphi(x)={\lambda_0^2}\varphi$, we obtain:
\begin{eqnarray*}&&\int_0^t\int_{\mathbb{R}^n}u_{tt}\psi dxds+\int_{0}^t\int_{\mathbb{R}^n}u\varphi(-\lambda_0^2\rho-a'(s)\rho-a(s)\rho'+b\rho)dxds\\
&&+\int_{\mathbb{R}^n}a(s)\psi udx\bigg|_0^t =\int_0^t\int_{\mathbb{R}^n}|u|^p\psi dxds.
\end{eqnarray*}
Since \eqref{rho-eqn}, the above equation simplifies to
$$\int_0^t\int_{\mathbb{R}^n}u_{tt}\psi dxds-\int_{0}^t\int_{\mathbb{R}^n}u\varphi\rho''dxds+\int_{\mathbb{R}^n}a(s)\psi udx\bigg|_0^t =\int_0^t\int_{\mathbb{R}^n}|u|^p\psi dxds.$$
Thus the integration by parts gives
$$\int_{\mathbb{R}^n}(u_t\psi-u\psi_t+a(s)u\psi) dx\bigg|_0^t=\int_0^t\int_{\mathbb{R}^n}|u|^p\psi dxds.$$
As the righthand side integral is positive, we obtain
$$G_1'(t)+\big(a(t)-2\frac{\rho'(t)}{\rho(t)}\big)G_1(t)\geq\varepsilon \int_{\mathbb{R}^n}\bigg(\rho(0)g(x)+(a(0)\rho(0)-\rho'(0))f(x)\bigg)\varphi(x)dx.$$
Denote
\begin{equation}
C_{f,g}:=\int_{\mathbb{R}^n}\bigg(g(x)+(a(0)-\rho'(0))f(x)\bigg)\varphi(x)dx,\label{C_fg}
\end{equation}
then by the compact support of $g(x)$ and $f(x)$ and assumption \eqref{ini2}, $C_{f,g}$ is finite and positive. We come to the differential inequality of $G_1$
$$G_1'(t)+\big(a(t)-2\frac{\rho'(t)}{\rho(t)}\big)G_1(t)\geq\varepsilon C_{f,g}.$$
Setting $A(t)=e^{\int_0^ta(\tau)d\tau}$, multiplying $\frac{A(t)}{\rho^2(t)}$ on two sides and then integrating over $[0,t]$, we derive
the lower bound of $G_1$
\begin{equation}\label{G1}
G_1(t)\geq\varepsilon C_{f,g}\frac{\rho^2(t)}{{A(t)}}\int_0^t\frac{{A(s)}}{\rho^2(s)}ds\geq\varepsilon C_{f,g,a}\int_0^t\frac{\rho^2(t)}{\rho^2(s)}ds.
\end{equation}
On the other hand, the denominator of \eqref{factor} can be estimated by using \eqref{2sided} and integration by parts,
\begin{eqnarray}
&&\int_{|x|\leq t+R}|\psi|^{p'}(t,x)dx = |\rho|^\frac{p}{p-1}(t)\int_{|x|\leq t+R}\varphi_{\lambda_0}^{p'}(x)dx\nonumber\\
&&\leq|\rho|^\frac{p}{p-1}(t)\lambda_0^{-n}\int_{\lambda_0|x|\leq\lambda_0(t+R)}D_1^{p'}\langle\lambda_0|x| \rangle^{-\frac{n-1}2p'} e^{p'\lambda_0|x|}d\lambda_0x\nonumber\\
&&\leq|\rho|^\frac{p}{p-1}(t)\cdot C_{\varphi}\lambda_0^{-1-\frac{n-1}2\frac{p}{p-1}}(t+R)^{n-1-\frac{n-1}2\frac{p}{p-1}}e^{\frac{p}{p-1}\lambda_0(t+R)}.\label{deno}
\end{eqnarray}
Combing the estimate \eqref{G1}, \eqref{deno} and \eqref{factor}, we now have
\begin{eqnarray*}\int_{\mathbb{R}^n}|u(x,t)|^pdx&\geq&\frac{\varepsilon^p C^p_{f,g,a}|\rho|^{p}(t)(\int_0^t\frac{1}{\rho^2(s)}ds)^p}
{C^{p-1}_{\varphi,R}\lambda_0^{-(p-1)-\frac{n-1}{2}p}(1+t)^{(n-1)(p-1)-\frac{n-1}{2}p}e^{\lambda_0p(t+R)}}\\
&>&C_0\varepsilon^p(1+t)^{(n-1)(1-\frac p2)}|\rho|^{p}(t)(\int_0^t\frac{1}{\rho^2(s)}ds)^pe^{-p\lambda_0t}
\end{eqnarray*}
Since $\rho(t)\sim e^{-\lambda_0t}$ for large $t$, the estimate \eqref{low:u^p} is obtained.
\end{proof}
\begin{remark}As matter of fact, this lemma can be also inferred from Lemma \ref{lem:aux} and Proposition \ref{prop:identity}, see \cite{WY}.
We here give an independent proof which does not rely on any sign of $a(t)$ and $b(t)$.
\end{remark}
Now, we are in the position to give the proof of Theorem \ref{thm:sub}. First utilizing the compact support of solution, applying the H\"{o}lder inequality on \eqref{low:G(t)} gives:
\begin{eqnarray}
G(t)&\geq& C_{r_1,r_2}\int_0^t\int_0^{s_2}\int_{\mathbb{R}^n}|u|^pdxds_1ds_2\nonumber\\
&\geq& C_{r_1,r_2}\int_0^tds_2\int_0^{s_2}G^p(s_1)(s_1+R)^{(1-p)n}ds_1.\label{iter-1}
\end{eqnarray}
Second, combining \eqref{low:G(t)} and \eqref{low:u^p}, we have:
\begin{eqnarray}
G(t)&\geq& C_{r_1,r_2}\int_0^t ds_2 \int_0^{s_2} C_0\varepsilon^p \langle t\rangle ^{(n-1)(1-p/2)} ds_1\nonumber\\
&\geq &C_0C_{r_1,r_2} \ep^p (1+t)^{-(n-1)p/2}\int_0^tds_2\int_0^{s_2}s_1^{n-1}ds_1\nonumber\\
&=&C_2\ep^p(1+t)^{-(n-1)p/2}t^{n+1},\label{iter-0}
\end{eqnarray}
where $C_2=\frac{C_0C_{r_1,r_2}}{n(n+1)}$
Then following same iteration argument of Theorem 2.1 in \cite{LT}, the blow-up result with the lifespan upper bound for sub-critical case can be obtained. For reader's convenience, we give the details here. Specifically, we assume that $G(t)$ satisfies the inequalities
\begin{equation}\label{iter-plug}
G(t)>D_j(1+t)^{-a_j}t^{b_j}\ \ \mbox{for}\ \ t\geq 0
\end{equation}
with positive constants $D_j,\ a_j,\ b_j,\ j\in\mathbb{N}$. Plugging \eqref{iter-plug} into \eqref{iter-1}, we obtain:
$$G(t)>\frac{C_{r_1,r_2}D_j^p}{(pb_j+2)^2}(1+t)^{-{n(p-1)-pa_j}}t^{pb_j+2},$$
where the above constants $\{a_j\},\ \{b_j\},\ \{D_j\}$ can be chosen to satisfy the iteration scheme:
$$a_{j+1}=pa_j+n(p-1),\ b_{j+1}=pb_j+2\ \ \mbox{and}\ \ D_{j+1}= \frac{C_{r_1,r_2}D_j^p}{(pb_j+2)^2}.$$
The initial case $j=1$ is derived from \eqref{iter-0}. We may take
$$a_1=(n-1)\frac{p}2,\ \ b_1=n+1\ \ \mbox{and}\ D_1=C_2\ep^p.$$
Hence, $a_j$ and $b_j$ can be derived directly
\begin{equation}a_j=p^{j-1}\left((n-1)\frac{p}2+n\right)-n,\ b_j=p^{j-1}(n+1+\frac2{p-1})-\frac2{p-1}.\label{a_jb_j}\end{equation}
Moreover, $$ D_{j+1}=\frac{C_{r_1,r_2}D_j^p}{b_{j+1}^2}\geq C_3\frac{D_j^p}{p^{2j}}.$$
where $C_3=\frac{C_{r_1,r_2}}{(n+1+\frac2{p-1})^2}$.
Taking logarithm on both sides and the iteration on $\log D_j$ gives:
\begin{eqnarray*}
\log D_j&\geq& p\log D_{j-1}-2(j-1)\log p+\log C_3\\
&\geq& p^2\log D_{j-2}-2(p(j-2)+(j-1))\log p+(p+1)\log C_3\\
&\geq&\cdots\\
&\geq&p^{j-1}\log D_1-2\log p\sum_{k=1}^{j-1}kp^{j-1-k}+\log C_3\sum_{k=1}^{j-1}p^k.
\end{eqnarray*}
As
$$\sum_{k=1}^{j-1}kp^{j-1-k}=\frac{1}{p-1}(\frac{p^j-1}{p-1}-j)\ \ \mbox{and}\ \ \sum_{k=1}^{j-1}p^k=\frac{p-p^j}{1-p},$$
we have
\begin{eqnarray*}
\log D_j&\geq& p^{j-1}\log D_1-\frac{2\log p}{p-1}\left(\frac{p^j-1}{p-1}-j\right)+\log C_3\frac{p-p^j}{1-p}\\
&=&p^{j-1}\bigg(\log D_1-\frac{2p\log p}{(p-1)^2}+\frac{p\log C_3}{p-1}\bigg)+\frac{2\log p}{p-1}j+\frac{2\log p}{(p-1)^2}+\frac{p\log C_3}{1-p}.
\end{eqnarray*}
Consequently, for $j>\left[\frac{p\log C_3}{2\log p}-\frac{1}{p-1}\right]+1$,
\begin{equation}\label{D_j}D_j\geq\exp\left\{p^{j-1}(\log D_1-S_p(\infty))\right\}\end{equation}
where
$$ S_p(\infty):=\frac{2p\log p}{(p-1)^2}-\frac{p\log C_3}{p-1}.$$
Inserting \eqref{a_jb_j} for $a_j,\ b_j$ and \eqref{D_j} for $D_j$ into \eqref{iter-plug} gives
\begin{eqnarray}
G(t)&\geq&\exp\left(p^{j-1}(\log D_1-S_p(\infty))\right)(1+t)^{-\alpha p^{j-1}+n}t^{\beta p^{j-1}-\frac{2}{p-1}}\nonumber\\
&\geq&\exp\big(p^{j-1}J(t)\big)(1+t)^{n}t^{-\frac{2}{p-1}}\label{contr}
\end{eqnarray}
where we denote
$$\alpha:=(n-1)\frac{p}2+n,\ \beta:=n+1+\frac2{p-1}$$
and
$$J(t):=\log D_1-S_p(\infty)-\alpha\log(1+t)+\beta\log(t).$$
For $t>1$, we have
\begin{eqnarray*}
J(t)&\geq& \log D_1-S_p(\infty)-\alpha\log(2t)+\beta\log(t)\\
&\geq&\log D_1-S_p(\infty)+(\beta-\alpha)\log(t)-\alpha\log2\\
&=&\log(D_1\cdot t^{\beta-\alpha})-S_p(\infty)-\alpha\log2.
\end{eqnarray*}
Noticing that $$\beta-\alpha=\frac{\gamma(p,n)}{2(p-1)},$$
thus if
$$t>\max\left(\left(\frac{e^{[S_p(\infty)+\alpha\log2]+1}}{C_2\varepsilon^p}\right)^{\frac{2(p-1)}{\gamma(p,n)}},1\right),$$
we then get $J(t)>1$, and this in turn give that $G(t)\rightarrow\infty$ by taking $j\rightarrow\infty$ in \eqref{contr}.
Therefore, for $\varepsilon<\varepsilon_0$, we obtain the desired upper bound,
$$T\leq C_4\varepsilon^{-\frac{2p(p-1)}{\gamma(p,n)}}$$
with
$$C_4:=\left(\frac{e^{(S_p(\infty)+\alpha\log2)+1}}{C_2}\right)^{2(p-1)/\gamma(p,n)}.$$
This completes our proof of Theorem $1.3$.

\section{Critical case: Proof of Theorem \ref{thm:cri}}\label{pf:cri}

In this section, we focus on the critical case. We first introduce some auxiliary functions. Given $\lambda_0\in (0,\alpha/2]$ and $q>-1$, let
\begin{eqnarray}
\label{aq}
\xi_q (x,t) & = & \int_{0}^{\lambda_0}e^{-\lambda(t+R)}\cosh\lambda t \: \varphi_\lambda(x)\lambda^{q}
d\lambda,\\
\eta _q (x,t,s) & = & \int_{0}^{\lambda_0}e^{-\lambda(t+R)}\frac{\sinh \lambda(t-s)}{\lambda(t-s)}
\: \varphi_\lambda(x)\lambda^{q} d\lambda,
\label{bq}
\end{eqnarray}
for $(x,t)\in \mathbb{R}^n\times \mathbb{R}$ and $s\in\mathbb{R}.$ Useful estimates are collected in the next lemma.

\begin{lemma}
\label{lem:aux}
Let $n\geq 2$. There exists $\lambda_0\in (0,\alpha/2]$, such that the following hold:

(i) if $0<q$, $|x|\leq R$ and $0\leq t$, then
\begin{eqnarray*}
\xi_q (x,t) & \geq & A_0,\\
\eta_q (x,t,0) & \geq & B_0\langle t\rangle^{-1};
\end{eqnarray*}

(ii) if $0<q$, $|x|\leq s+R$ and $0\leq s<t$, then
\begin{eqnarray*}
\eta_q (x,t,s) & \geq & B_1 \langle t\rangle^{-1}\langle s\rangle^{-q};
\end{eqnarray*}

(iii) if $(n-3)/2<q$, $|x|\leq t+R$ and $0<t$, then
\begin{eqnarray*}
\eta_q (x,t,t) & \leq & B_2 \langle t\rangle^{-(n-1)/2}\langle t-|x| \rangle^{(n-3)/2-q}.
\end{eqnarray*}
Here $A_0$ and $B_k$, $k=0,1,2,$ are positive constants depending only on $\alpha$, $q$ and $R$,
while $\langle s\rangle =3+|s|$.
\end{lemma}

\begin{proof}
See Lemma 3.1 in \cite{WY}.
\end{proof}
Next, we shall consider the lower bound of averaged functionals
$$F(t)=\int_{\mathbb{R}^n}u(x,t)\varphi_\lambda(x)dx.$$
Testing equation \eqref{main} by $\varphi_\lambda(x)$ to derive:
\begin{equation}\label{F(t)}
F''(t)+a(t)F'(t)+(b(t)-\lambda^2)F(t)=\int_{\mathbb{R}^n}|u|^p\varphi_\lambda(x)dx.
\end{equation}
Define the operator:
$$L_{a,b}:=\partial_t^2+a(t)\partial_t+b(t)-\lambda^2,$$
and
$$L_{r_1,r_2}:=\partial_t^2+(r_1(t)-r_2(t))\partial_t-\lambda^2.$$
Supposing $L_{a,b}y(t,s;\lambda)=0$, by utilizing \eqref{r_2}, it is easy to find that $u(t,s,\lambda):=y(t,s;\lambda)e^{A_2(t,s)}$ satisfies:
\begin{equation}\label{u}
L_{r_1,r_2}u=0.
\end{equation}
Moreover, under the assumption of $a(t)$ and $b(t)$ for critical case in Theorem \ref{thm:cri}, $a(t),\ b(t)\geq0$, then by Lemma \ref{lem:L^1}, $r_1(t)-r_2(t)=a(t)-2r_2(t)$ is positive for large $t$. Hence the fundamental solution system $\{\chi_1(t,s;\lambda),\chi_2(t,s;\lambda)\}$ of \eqref{u} can be found with the aid of lemma 2.3 in \cite{WY-damped}, which satisfy the initial conditions
 $$L_{r_1,r_2}\chi_1(t,s;\lambda)=0,\ \ \ \chi_1(s,s;\lambda)=1,\ \ \ \partial_t\chi_1(s,s;\lambda)=0,$$
 $$L_{r_1,r_2}\chi_2(t,s;\lambda)=0,\ \ \ \chi_2(s,s;\lambda)=0,\ \ \ \partial_t\chi_2(s,s;\lambda)=1,$$
 and the asymptotic estimates for $t\geq s\geq 0$:
\begin{eqnarray*}
(i) && \cosh\lambda(t-s)\geq \chi_1(t,s;\lambda) \geq e^{-\|r_1-r_2\|_{L^1}}\cosh\lambda(t-s),\\
(ii) &&  e^{\|r_1-r_2\|_{L^1}}\frac{\sinh \lambda(t-s)}{\lambda}\geq \chi_2(t,s;\lambda) \geq e^{-2\|r_1-r_2\|_{L^1}}\frac{\sinh \lambda(t-s)}{\lambda}.
\end{eqnarray*}
Consequently, we find the fundamental solution of $L_{a,b}y=0$, by setting
$$y_1(t,s,\lambda)=[r_2(s)\chi_2(t,s;\lambda)+\chi_1(t,s;\lambda)]e^{-A_2(t,s)}$$
$$y_2(t,s;\lambda)=\chi_2(t,s;\lambda)e^{-A_2(t,s)}.$$
Returning to \eqref{F(t)}, the Duhamel's principle gives
\begin{eqnarray*}
F(t) &=& y_1(t,0;\lambda)F(0) + y_2(t,0;\lambda)F'(0) + \int_0^t y_2(t,s;\lambda)\int_{\mathbb{R}^n}|u(x,s)|^p\varphi_\lambda(x)dxds\\
&=& \chi_1(t,0;\lambda)e^{-A_2(t,0)}F(0)+\chi_2(t,0;\lambda)e^{-A_2(t,0)}[r_2(0)F(0)+F'(0)]\\
&&+\int_0^t\chi_2(t,s;\lambda)e^{-A_2(t,s)}\int_{\mathbb{R}^n}|u(x,s)|^p\varphi_\lambda(x)dxds.
\end{eqnarray*}
Making use of asymptotic property of $\chi_1,\ \chi_2$ and the non-negativity of $f(x)$ and $g(x)+r_2(0)f(x)$, we have
\begin{eqnarray}
&&\int_{\mathbb{R}^n} u(x,t)\varphi_\lambda(x)dx\geq \varepsilon e^{-\|r_1-r_2\|_{L^1}-\|r_2\|_{L^1}}\cosh (\lambda t)\int_{\mathbb{R}^n} f(x)\varphi_{\lambda}(x)dx\nonumber\\
&&+\varepsilon e^{-2\|r_1-r_2\|_{L^1}-\|r_2\|_{L^1}}\frac{\sinh (\lambda t)}{\lambda}\int_{\mathbb{R}^n} [r_2(0)f(x)+g(x)]\varphi_{\lambda}(x)dx\\
&&+e^{-2\|r_1-r_2\|_{L^1}-\|r_2\|_{L^1}}\int_0^t \frac{\sinh (\lambda(t-s))}{\lambda}\int_{\mathbb{R}^n} |u(x,s)|^p\varphi_\lambda(x)dxds.\nonumber
\end{eqnarray}
Multiplying the above inequality by $\lambda^q e^{-\lambda(t+R)}$, integrating on $[0,\lambda_0]$ and interchanging the order of integration between $\lambda$ and  $x$, we get the following Proposition.
\begin{proposition}
\label{prop:identity}
Let $q>-1$ and the assumptions of Theorem \ref{thm:cri} be fulfilled. That is  $a(t),\ tb(t)\in L^1([0,\infty))$, $a(t)\in C^1([0,\infty))$, $b(t)\in C([0,\infty))$ and $f(x)$ and $r_2(0)f(x)+g(x)$ are non-negative. Then we have following lower bound estimate,
\begin{equation}
\label{final-equal}
\begin{array}{lll}
\displaystyle \int_{{\mathbb{R}}^n}u(x,t) \eta_{q}(x,t,t) dx \ge \varepsilon e^{-\|r_1-r_2\|_{L^1}-\|r_2\|_{L^1}}\!\!\!\int_{{\mathbb{ R}^n}}\!\!f(x)\xi_{q}(x,t)\: dx\\
\displaystyle+ \varepsilon e^{-2\|r_1-r_2\|_{L^1}-\|r_2\|_{L^1}}t\int_{{\mathbb{R}}^n}\!\!(r_2(0)f(x)+g(x)) \eta_{q}(x,t,0) dx\\
\displaystyle+e^{-2\|r_1-r_2\|_{L^1}-\|r_2\|_{L^1}}\!\!\!\int_0^t(t-s) \int_{{\mathbb{R}}^n}|u(x,s)|^p \eta_{q}(x,t,s) dxds
\end{array}
\end{equation}
for all $t\in (0,T_\varepsilon).$
\end{proposition}

We further define the functional
\begin{equation}
\widetilde{F}(t):=\int_{\mathbb{R}^n} u(x,t) \eta_{q}(x,t,t)dx,\ q>-1.
\end{equation}
Combing Proposition \ref{prop:identity} and Lemma \ref{lem:aux}, we obtain the convenient
iteration frame Proposition \ref{prop:frame} by same argument as proof of Proposition 4.2 in \cite{WY}.
\begin{proposition}
\label{prop:frame}
Suppose that the assumptions in Theorem \ref{thm:cri} are fulfilled and choose
$q=(n-1)/2-1/p.$ There exists a positive constant $C=C(n,p,R,r_1,r_2)$, such that
\begin{equation}
\label{frame}
\widetilde{F}(t)  \geq \frac{C}{\langle t\rangle}
\int_0^t \frac{t-s}{\langle s\rangle}\frac{\widetilde{F}(s)^p}{
(\log \langle s\rangle)^{p-1}}\: ds
\end{equation}
for all $t\in (0,T_\varepsilon).$
\end{proposition}
Now, Theorem \ref{thm:cri} can be proved by using the iteration arguments  same as the proof in Sections 4 and 5 in \cite{WY}.
For reader's convenience, we give the details here. The initial step of iteration argument starts from the Proposition \ref{prop:identity}. Neglecting the first two positive initial terms of \eqref{final-equal} and inserting the lower bound estimates \eqref{low:u^p} for $L^p$ norm and (ii) with $q=(n-1)/2-1/p$ for $\eta_q(x,t,s)$ in Lemma \ref{lem:aux}, we obtain:
\begin{eqnarray*}
&&\int_{{\mathbb{R}}^n}u(x,t) \eta_{q}(x,t,t) dx \\
&\ge& e^{-2\|r_1-r_2\|_{L^1}-\|r_2\|_{L^1}}\int_0^t(t-s) \int_{{\mathbb{R}}^n}|u(x,s)|^p \eta_{q}(x,t,s) dxds\\
&\ge&\frac{C_0B_1e^{-2\|r_1-r_2\|_{L^1}-\|r_2\|_{L^1}}\ep^p}{\langle t\rangle}\int_0^t\frac{t-s}{\langle s\rangle ^{q+(n-1)p/2-(n-1)}}ds.
\end{eqnarray*}
Notice that $$q+(n-1)p/2-(n-1)=\frac{(n-1)p}{2}-\frac{n-1}2-\frac1p=1.$$
Denote $M=C_0B_1e^{-2\|r_1-r_2\|_{L^1}-\|r_2\|_{L^1}}$, for $t>\frac32$,
\begin{eqnarray}
\widetilde{F}(t)&\ge&\frac{M\ep^p}{t}\int_1^t\frac{t-s}{s}ds=\frac{M\ep^p}{t}\int^t_{1}\log sds\nonumber\\
&\ge&\frac{M\ep^p}{t}\int^t_{2t/3}\log sds\ge \frac{M}3\ep^p\log(2t/3)\label{slicing-0}.
\end{eqnarray}
Next, we apply the ``slicing method'' under the iteration frame Proposition \ref{prop:frame}. We have following Proposition for $\widetilde{F}(t)$.
\begin{proposition}
Suppose the assumption of Theorem \ref{thm:cri} are fulfilled, then
\begin{equation}\label{slicing-1}
\widetilde{F}(t)\ge C_j(\log \langle t\rangle)^{-b_j}(\log(t/l_j))^{a_j} \ \ \mbox{for}\ t\geq l_j
\end{equation}
with series of constants:
\begin{equation}\label{cri-constants}\ l_j=2-2^{-(j+1)},\ \ a_j=\frac{p^{j+1}-1}{p-1},\ \ b_j=p^{j}-1\end{equation}
and $C_1=N\ep^{p^2}$, where $N=\frac{CM^p}{3^p7}$. For $j\geq 2$,
\begin{equation}\label{C_j}
C_{j}=\exp\left(p^{j-1}\log\left( C_1(2p)^{-\frac{p}{p-1}}E^\frac1{p-1}\right)-\log E^{\frac1{p-1}}+\log(2p)^{\frac{p}{p-1}+j-1}\right),\end{equation}
where $E=\frac{C(p-1)}{8p^2}$.
\end{proposition}
\begin{proof}
For case $j=1$, inserting \eqref{slicing-0} into \eqref{frame} and replacing the domain of integral by $[l_0,t]$, we have for $t\geq l_1$,
\begin{eqnarray*}
\widetilde{F}(t)&\geq&\frac{C}{\langle t\rangle}
\int_{l_0}^t \frac{t-s}{\langle s\rangle}\frac{(M/3\ep^p\log(2t/3))^p}{
(\log \langle s\rangle)^{p-1}}\: ds\\
&\geq&\frac{CM^{p}\ep^{p^2}}{3^pt}\log\langle t\rangle^{1-p}\int_{l_0}^t\frac{t-s}{s}\log^p(s/l_{0})ds\\
&\geq&\frac{CM^{p}\ep^{p^2}}{3^pt}\log\langle t\rangle^{1-p}\int_{l_0}^t\log(s/l_{0})^{p+1}ds
\end{eqnarray*}
the last line is due to integration by parts. Since, $t\geq l_1$, then $\frac{l_0}{l_1}t\geq l_0$. Hence,
$$\int_{l_0}^t\log(s/l_{0})^{p+1}ds\geq \int_{\frac{l_0}{l_1}t}^t\log(s/l_{0})^{p+1}ds\geq \frac17t\log(t/l_1)^{p+1}.$$
Consequently, $$\widetilde{F}(t)\geq\frac{CM^{p}\ep^{p^2}}{3^p7}(\log \langle t\rangle)^{1-p}(\log(t/l_1))^{p+1}.$$
For the case of $j\geq 2$, we prove \eqref{slicing-1} by induction. Assume \eqref{slicing-1} holds for $j$ with $t\geq l_j$, inserting this into \eqref{frame} to obtain:
\begin{eqnarray*}
\widetilde{F}(t)&\geq&\frac{CC_j^p}{\langle t\rangle}
\int_{l_j}^t \frac{t-s}{\langle s\rangle}\frac{(\log \langle t\rangle)^{-b_jp}(\log(t/l_j))^{a_jp}}{
(\log \langle s\rangle)^{p-1}}\: ds\\
&\geq&\frac{CC_j^p}{t}\log\langle t\rangle^{-b_jp+1-p}\int_{l_j}^t\frac{t-s}{s}(\log(s/l_j))^{a_jp}ds\\
&\geq&\frac{CC_j^p}{t(a_jp+1)}\log\langle t\rangle^{-b_jp+1-p}\int_{l_j}^t\log(s/l_{j})^{a_jp+1}ds.
\end{eqnarray*}
Let $t\geq l_{j+1}$, replacing the domain of integration by $[l_j/l_{j+1}t,t]$ and noticing that \eqref{cri-constants} for $a_j,\ b_j$ and $1-l_j/l_{j+1}$, we get
\begin{eqnarray*}
\widetilde{F}(t)&\geq&\frac{CC_j^p}{t(a_jp+1)}\log\langle t\rangle^{-b_jp+1-p}\int_{l_j/l_{j+1}t}^t\log(s/l_{j})^{a_jp+1}ds\\
&\geq&\frac{C(1-l_j/l_{j+1})C_j^p}{a_jp+1}\log\langle t\rangle^{-b_jp+1-p}\log(t/l_{j+1})^{a_jp+1}\\
&\geq&\frac{CC_j^p}{2^{j+3}a_{j+1}}\log\langle t\rangle^{-b_{j+1}}\log(t/l_{j+1})^{a_{j+1}}\\
&\geq&\frac{C(p-1)C_j^p}{2^{j+3}p^{j+2}}\log\langle t\rangle^{-b_{j+1}}\log(t/l_{j+1})^{a_{j+1}}.
\end{eqnarray*}
Let
$$C_{j+1}=\frac{C(p-1)}{8p^2}\frac{C_{j}^p}{(2p)^{j}},$$ by directly induction of this relation, we have
\begin{eqnarray*}
&&\log C_{j+1}=p\log C_j+\log\frac{C(p-1)}{8p^2}-j\log(2p)\\
=&&p\left(p\log C_{j-1}+\log\frac{C(p-1)}{8p^2}-(j-1)\log(2p)\right)+\log\frac{C(p-1)}{8p^2}-j\log(2p)\\
=&&p^2\log C_{j-1}+(p+1)\log\frac{C(p-1)}{8p^2}-\left(p(j-1)+j\right)\log(2p)\\
=&&\cdots\\
=&&p^{j}\log {C_1}+\sum_{l=0}^{j-1}p^l\log\frac{C(p-1)}{8p^2}-\sum_{l=0}^{j-1}p^l(j-l)\log(2p)\\
=&&p^{j}\log {C_1}+\frac{1-p^j}{1-p}\log\frac{C(p-1)}{8p^2}-p^j\sum_{k=1}^{j}kp^{-k}\log(2p)\\
=&&p^{j}\log {C_1}+\frac{1-p^j}{1-p}\log\frac{C(p-1)}{8p^2}-(\frac{p^{j+1}-p}{p-1}-j)\log(2p)\\
=&&p^j\log\left( C_1(2p)^{-\frac{p}{p-1}}(\frac{C(p-1)}{8p^2})^\frac1{p-1}\right)-\frac1{p-1}\log\frac{C(p-1)}{8p^2}\\
&&+(\frac{p}{p-1}+j)\log(2p)
\end{eqnarray*}
which is just \eqref{C_j}.
\end{proof}
From \eqref{slicing-1}, \eqref{a_jb_j} and \eqref{C_j}, we find for $t>2$
\begin{eqnarray*}&&\widetilde{F}(t)\ge C_j(\log \langle t\rangle)^{-b_j}(\log(t/l_j))^{a_j}\\
&&>E^{-\frac1{p-1}}\exp\left(p^{j-1}\log\left( C_1(2p)^{-\frac{p}{p-1}}E^\frac1{p-1}\right)\right)(\log \langle t\rangle)^{1-p^j}(\log(t/l_j))^{\frac{p^{j+1}-1}{p-1}}\\
&&=E^{-\frac1{p-1}}\left( C_1(2p)^{-\frac{p}{p-1}}E^\frac1{p-1}\right)^{p^{j-1}}(\log \langle t\rangle)^{1-p^j}(\log(t/l_j))^{\frac{p^{j+1}-1}{p-1}}\\
&&>E^{-\frac1{p-1}}\left( C_1(2p)^{-\frac{p}{p-1}}E^\frac1{p-1}(\log \langle t\rangle)^{-p}  (\log(t/2))^{\frac{p^{2}}{p-1}}  \right)^{p^{j-1}}(\log \langle t\rangle)(\log(t/2))^{\frac{-1}{p-1}}.
\end{eqnarray*}
We concern on the $C_1(2p)^{-\frac{p}{p-1}}E^\frac1{p-1}(\log \langle t\rangle)^{-p}  (\log(t/2))^{\frac{p^{2}}{p-1}}$. For $t>4$, we have
$$\log \langle t\rangle<2\log(t),\ \log(t/2)>\frac12\log(t),$$
so that
\begin{eqnarray*}
&&C_1(2p)^{-\frac{p}{p-1}}E^{\frac1{p-1}}(\log \langle t\rangle)^{-p}  (\log(t/2))^{\frac{p^{2}}{p-1}}\\
&\geq&C_1(2p)^{-\frac{p}{p-1}}E^{\frac1{p-1}}2^{-p-\frac{p^2}{p-1}} \log(t)^{\frac{p}{p-1}}\\
&\geq&N2^{-\frac{2p^2}{p-1}}p^{-\frac{p}{p-1}}E^{\frac1{p-1}} \ep^{p^2}\log(t)^{\frac{p}{p-1}}\\
&:=&B\ep^{p^2}\log(t)^{\frac{p}{p-1}},
\end{eqnarray*}
where we denote $B=N2^{-\frac{2p^2}{p-1}}p^{-\frac{p}{p-1}}E^{\frac1{p-1}}$.
Take small $\ep_0=\ep_0(u_0,p,R,n,r_1,r_2)$ so that
$$\exp\{B^{-(p-1)/p}\ep_0^{-p(p-1)}\}\geq4,$$ then for $\ep<\ep_0$, supposing that $T$ satisfies
$$T>\exp\{B^{-(p-1)/p}\ep^{-p(p-1)}\}(\geq4),$$
we have  $$C_1(2p)^{-\frac{p}{p-1}}E^\frac1{p-1}(\log \langle t\rangle)^{-p}  (\log(t/2))^{\frac{p^{2}}{p-1}}>1,$$
which gives contradiction by taking $j\rightarrow\infty$.

\section{Appendix}

In this appendix, we outline the proof of local existence and finite propagation speeds, which is mainly based on the proof of \cite{Lai-Sch-Taka}.
Also, we give the proof of Lemma \ref{lem:L^1} and Lemma \ref{lem:rho}, both of which relate with the solution behavior of second order ODE with variable coefficient.

\textit{Proof of local existence and finite propagation speeds:}

Basically, the construction of the local solutions is same as the
argument in \cite{Lai-Sch-Taka}, but with minor modification due to the perturbation of Laplacian.
For readers' convenience, we give the details of proof.

From the assumption \eqref{g}, we know that there exists a positive constant $C_1$ such that
\begin{equation}
\label{g-est}
|g_{ij}(x)|\le 1+ C_1e^{-\alpha|x|}
\end{equation}
holds for $i,j=1,2\cdots,n$.
Let
\begin{equation}
S(r)=r+\frac{1}{\alpha}\log(1+C_1e^{-\alpha r}),\ r\geq0
\end{equation}
and denote $S(x)=S(|x|)$,
$R=S(R_0)=R_0+\frac{1}{\alpha}\log(1+C_1e^{-\alpha R_0})$.
A direct calculation gives $S'(r)=\frac{1}{1+C_1e^{-\alpha r}}<1$ so that $S(x)$ satisfies the triangle inequality, as
\begin{equation}S(|x|)-S(|y|)=S'(\theta)\big||x|-|y|\big|<|x-y|<S(|x-y|).\label{trian}\end{equation}
We may also find by combining \eqref{g-est} that
$$\sum_{i,j=1}^ng_{ij}(x)|\partial_{x_i}S(x)| |\partial_{x_j} S(x)|\le \frac{1}{1+C_1e^{-\alpha|x|}}$$
which simply implies that $S(x)$ satisfies following differential inequality
\begin{eqnarray}
\label{ineq-S}
&&\sum_{i,j=1}^{n}g_{ij}(x) \partial_{x_i}S(x)\partial_{x_j} S(x)\le1.
\end{eqnarray}
Moreover, as $S(x)$ is increasing function with respect to $|x|$, so that
\begin{equation}
\label{support-R}
\{x\in\mathbb{R}^n:\ |x|\le R_0 \}\subset \{x \in \mathbb{R}^n:\ S(x) \leq R\}.
\end{equation}
We now take our function space as following Banach space $X(T)$:
\begin{eqnarray*}
X(T)& = & \{\phi\in C([0,T),H^1(\mathbb{R}^n))\cap C^1([0,T),L^2(\mathbb{R}^n)):\\
 &&\quad \hbox{supp}\: \phi  \subset B_S(t+R),\ \|\phi\|_{X(T)}\le K \},\\
\|\phi\|_{X(T)} & = & \sup_{t\in[0,T)}E^{1/2}(t), \ \ E(t)=\int_{\mathbb{R}^n}
((\partial_t\phi)^2+\sum_{i,j=1}^{n}g_{ij}(x)\partial_{x_i}\phi \partial_{x_j}\phi)dx,
\end{eqnarray*}
where $B_S(t+R)=\{x: S(x)\leq t+R\}$ and $K,\ T$ are positive constants.

Consider the following initial value problem for
$v\in X(T)$:
\begin{equation}
\left\{
\begin{array}{llll}
u_{tt}-\Delta_g u=H_v(x,t)\  \mbox{in}\ \mathbb{R}^n\times(0,\infty),\label{u=Mv}\\
u(x,0)=\varepsilon f(x),\
u_t(x,0)=\varepsilon g(x)\  \mbox{for}\ x\in\mathbb{R}^n,\nonumber
\end{array}
\right.
\end{equation}
where $H_v(x,t)=|v(x,t)|^p-a(t)v_t(x,t)-b(t)v(x,t)$.
We aim to show the map:
$$M: v\rightarrow u = Mv,\ v\in X(T)$$
is a contraction. In the following, we assume that $p \leq n/(n-2)$ when $n\geq 3$ and the positive constant $C$ may
vary from line to line.
For $v\in X(T)$, combining the H\"{o}lder's inequality and Gagliardo-Nirenberg inequality, we know that
$$\|v\|_{L^{2p}(\mathbb{R}^n)}\leq C\|v\|^{1-\theta(2p)}_{L^2(\mathbb{R}^n)}\|\nabla v\|^{\theta(2p)}_{L^2(\mathbb{R}^n)},\ \ \theta(2p):=\frac{n(p-1)}{2p}. $$
Moreover, as $|x|\leq S(x)<t+R$ in the support of $v$, the Poincar\'{e} inequality further gives
$$\|v\|_{L^2(\mathbb{R}^n)}\leq C(t+R)\|\nabla v\|_{L^2(\mathbb{R}^n)}.$$
Hence,
\begin{equation}\label{l2p}
\|v\|_{L^{2p}(\mathbb{R}^n)}\leq C(t+R)^{1-\theta(2p)}\|\nabla v\|_{L^2(\mathbb{R}^n)}\leq C(t+R)^{1-\theta(2p)}E_v^{1/2}.
\end{equation}
Consequently, we have $H_v(x,t)\in\ L^2(\mathbb{R}^n\times [0,T))$ for some fixed $T$.
By applying density argument in \cite{Struwe}(also see \cite{Lai-Sch-Taka}), it suffices to consider the smooth solution $u$ and compactly supported functions $f,g$.

We first show the finite propagation speed:
\begin{equation}
\label{finiteprop}
\supp u \subset \{(x,t)\in \mathbb{R}^n\times[0,T)\colon S(x)\leq t+R\}.
\end{equation}
Fix a point $(x_0, t_0)\in \mathbb{R}^n\times (0, T)$ such that $S(x_0)\geq t_0+R$ and set the backward cone with vertex at $(x_0,t_0)$
\begin{equation*}
	C := \{ (x,t) \in \mathbb{R}^n\times[0,T)\colon  S(x-x_0)\leq t_0-t \}.
\end{equation*}
Define the energy on the time-section of the cone as
\begin{equation}\label{local_ennorm}
e(t,u(t)):=\frac{1}{2}\int_{C_t}(u_t^2+\sum_{i,j=1}^{n}g_{ij}(x)\partial_{x_i}u \partial_{x_j}u)dx.
\end{equation}
Here
\begin{equation*}
	C_{t} := \{ x \in \mathbb{R}^n \colon S(x-x_0)\le t_0-t \},
\end{equation*}
which is outside the support of $v$, since
$$S(x)\geq S(x_0)-S(x-x_0)\geq (t_0+R)+(t-t_0)=t+R.$$
Proceeding as the proof of Theorem 8 in page 395 of \cite{Ev1}, we have:
\begin{eqnarray}
\frac{d}{dt}e(t,u(t))&=&\int_{C_t}u_tu_{tt}+\sum_{i,j=1}^n g_{ij}(x)u_{x_i}u_{x_jt}dx\nonumber\\&&
-\frac12\int_{\partial C_t}(u_t^2+\sum_{i,j=1}^n g_{ij}u_{x_i}u_{x_j})\frac{1}{|\nabla S(x-x_0)|}dS\label{A-B}:=A-B
\end{eqnarray}
Taking integration by parts for $A$, we have
\begin{equation*}
A=\int_{C_t}u_t(H_v-\sum_{i,j=1}^n g_{ijx_j}u_{x_i})dx+\int_{\partial C_t}\sum_{i,j=1}^n g_{ij}u_{x_i}\nu^ju_tdS
\end{equation*}
where $\nu^j$ is the $j$-th component of outer unit normal direction to $\partial C_t$. Note that on the $\partial C_t$, $S(x-x_0)=t-t_0$ so that $\nu=\frac{\nabla S(x-x_0)}{|\nabla S(x-x_0)|}$. By Cauchy-Schwarz inequality and \eqref{ineq-S}, we have
$$|\sum_{i,j=1}^n g_{ij}u_{x_i}\nu^j|\leq (\sum_{i,j=1}^n g_{ij}u_{x_i}u_{x_j})^\frac12(\sum_{i,j=1}^n g_{ij}\nu^i\nu^j)^\frac12\leq (\sum_{i,j=1}^n g_{ij}u_{x_i}u_{x_j})^\frac12\frac1{|\nabla S|}$$
which in turns that
$$A\leq \int_{C_t}u_t(H_v-\sum_{i,j=1}^n g_{ijx_j}u_{x_i})dx+\int_{\partial C_t}(\sum_{i,j=1}^n g_{ij}u_{x_i}u_{x_j})^\frac12u_t\frac1{|\nabla S|}dS.$$
Plugging this into \eqref{A-B} and utilizing \eqref{g}, we obtain:
\begin{eqnarray*}
\frac{d}{dt}e(t,u(t))&\leq&\int_{C_t}u_t(H_v-\sum_{i,j=1}^n g_{ijx_j}u_{x_i})dx\\
&\leq& C e(t,u(t))+ e^{1/2}(t,u(t))\|H_v(\cdot,t)\|_{L^2(C_t)},
\end{eqnarray*}
or equivalently,
\begin{equation*}
\frac{d}{dt}e^{1/2}(t,u(t)) \le Ce^{1/2}(t,u(t)) +
 \|H_{v}(\cdot,t)\|_{L^2(C_t)}.
\end{equation*}
As $C_t$ is outside the support of $v$, we have $H_v\equiv 0$. Using Gronwall's inequality, we finally obtain $e(t,u(t))=0$
for $0\le t \le t_0$ and hence $u\equiv 0$ in $C_t$.
Therefore, we get (\ref{finiteprop}).

Next, we show that $\|Mv\|_{X_t}\leq K$. We obtain the energy identity by multiplying $u_t$ at both sides of \eqref{u=Mv} such as (25) in \cite{Lai-Sch-Taka}:
\begin{eqnarray*}
&&\frac{\partial}{\partial t}\frac{1}{2}
\left((\partial_tu)^2+\sum_{i,j=1}^{n}g_{ij}(x)\partial_{x_i}u \partial_{x_j}u\right)\\
&&=\sum_{i,j=1}^{n}\partial_{x_i}(\partial_tug_{ij}(x)\partial_{x_j}u)+|v|^pu_t-a(t)vu_t-b(t)v_tu_t.
\end{eqnarray*}
Integrating the above identity over $\mathbb{R}^n\times[0,T]$ gives:
$$E(t,u(t))-E(0,u(0))=\int_0^tds\int_{\mathbb{R}^n}|v|^pu_s-a(s)vu_s-b(s)v_su_sdx.$$
As the uniform elliptic condition \eqref{uniform-g}, we have $E_{\phi}(t)\le E(t)$ for $E_{\phi}(t)$ defined in \cite{Lai-Sch-Taka}. Thus the terms in righthandside can be similar estimated. By exploiting \eqref{l2p}, we have
\begin{eqnarray*}
\int_{\mathbb{R}^n}|v|^pu_sdx&\leq& (\int_{\mathbb{R}^n}|v|^{2p}dx)^{1/2}E^{1/2}(s,u(s))\\
&\leq& C (s+R)^{p(1-\theta(2p))}E^{1/2}(s,v(s))E^{1/2}(s,u(s))
\end{eqnarray*}
and
$$\int_{\mathbb{R}^n}|v||u_t|dx\leq C(t+R)E^{1/2}(s,v(s))E^{1/2}(s,u(s)),$$
$$\int_{\mathbb{R}^n}|v_t||u_t|dx\leq 2E^{1/2}(s,v(s))E^{1/2}(s,u(s)).$$
Hence,
$$E(t,u(t))-E(0,u(0))\leq C \int_0^t a_K(s)E^{1/2}(s,u(s))ds,$$
where $$a_K(s)= K^p(s+R)^{p(1-\theta(2p))}+K(s+R)b(s)+Ka(s).$$
Applying the Bihari's inequality yields
\begin{eqnarray*}E^{1/2}(t,u(t))&\leq& E^{1/2}(0,u(0))+C\int_0^ta_K(s)ds\\
&\leq& E^{1/2}(0,u(0))+C\max\{K,K^p\}T(1+T)^\gamma
\end{eqnarray*}
for some positive $\gamma$. Picking a large $K$ and small enough $T$, so that $\|u\|_{X_T}\leq K$.

Finally, we show $M$ is a contraction map. Given $u_1=Mv_1$, $u_2=Mv_2$ and let $\overline{u}=u_1-u_2$ and $\overline{v}=v_1-v_2$, we observe that $\overline{u}$ satisfies:
\begin{equation}
\left\{
\begin{array}{llll}
\overline{u}_{tt}-\Delta_g \overline{u}=|v_1(x,t)|^p-|v_2(x,t)|^p-a(t)\overline{v}_t(x,t)-b(t)\overline{v}(x,t)\  \mbox{in}\ \mathbb{R}^n\times(0,\infty),\nonumber\\
\overline{u}(x,0)=0,\
\overline{u}_t(x,0)=0\  \mbox{for}\ x\in\mathbb{R}^n.\nonumber
\end{array}
\right.
\end{equation}
Multiplying $\overline{u}_t$ on both sides and proceeding similar as above, we finally reach:
$$\|\overline{u}\|_{X_T}\leq C\max\{1, K^{p-1}\}T(1+T)^\gamma\|\overline{v}\|_{X_T}.$$
Taking $T$ small enough so that $C\max\{1, K^{p-1}\}T(1+T)^\gamma\leq1$, we conclude with the desired result.
$ \hfill{} \Box$\\

\textit{Proof of Lemma \ref{lem:L^1}:}
We start the proof from the Ricatti's equation \eqref{r_2}.
Setting $r_2(t)=-\frac{k'(t)}{k(t)}$, then $k(t)$ satisfies:
$$k''+a(t)k'+b(t)k=0.$$
Moreover, let
$$c(t)=e^{\int_{t_0}^ta(\tau)d\tau},\ \ x =\int_{t_0}^t\frac 1{c(s)}ds,\ \ y(x)=k(t),$$
for some large $t_0$,
we find that:
\begin{equation}\label{ricatti}
\frac{d^2y}{dx^2}+c^2(t)b(t)y(x)=0.
\end{equation}
Since $e^{-\|a\|_{L^1}}<c(t)<e^{\|a\|_{L^1}}$ implies $$(t-t_0)e^{-\|a\|_{L^1}}<x<(t-t_0)e^{\|a\|_{L^1}}.$$
Integrating by changing of variable, for $\tilde{x}>0$ and $\tilde{t}>t_0$ ($\tilde{x} =\int_{t_0}^{\tilde{t}}\frac 1{c(s)}ds$),
\begin{eqnarray*}
\int_{\tilde{x}}^\infty c^2(t)b(t)dx &=& \int_{\tilde{t}}^\infty c^2(t)b(t) \frac1{c(t)}dt = \int_{\tilde{t}}^\infty c(t)b(t) dt\\
&<& e^{\|a\|_{L^1}}\int_{\tilde{t}}^\infty b(t)dt.
\end{eqnarray*}
Changing the integral variable and applying the mean value theorem for integrals, we find
$$\tilde{t}\int_{\tilde{t}}^\infty b(t)dt=\tilde{t}\int_0^{\frac1{\tilde{t}}}\frac1{s^2}b(\frac1s)ds=\frac1{\hat{s}^2}b(\frac1{\hat{s}}).$$
Here, $\hat{s}\in (0,\frac1{\tilde{t}})$ and we note that $tb(t)\in L^1$ implies $t^2b(t)\rightarrow 0$ as $t\rightarrow\infty$.
In this sense, we obtain
$$\int_{\tilde{x}}^\infty c^2(t)b(t)dx < e^{\|a\|_{L^1}}\int_{\tilde{t}}^\infty b(t)dt = o(\frac1{\tilde{t}})= o(\frac1{\tilde{x}}).$$
On the other hand, notice that equation $$y''+\frac1{4x^2}y=0$$
has the non-oscillating solution
$$y(x)=c_1\sqrt{x}+c_2\sqrt{x}\log(x)$$
and
$$\int_{\tilde{x}}^\infty \frac{1}{4x^2}dx=\frac1{4\tilde{x}}.$$
Hence, by Hille-Wintner comparison theorem (Theorem 2.12 in \cite{Swanson}), \eqref{ricatti} also has non-oscillation solution.
Furthermore, since $c(t)$ is finite and $(t-t_0)b(t)$ is in $L^1$, we have
$$\int_{0}^\infty(x-0)|c^2(t)b(t)|dx<\int_{t_0}^\infty(t-t_0)c(t)b(t)dt<\infty.$$
By Corollary 9.1 of Chapter XI in \cite{Hartman}, this implies that there exists a pair of solutions $y_1$ and $y_2$ of \eqref{ricatti}
satisfying, as $x\rightarrow \infty$,
\[
\begin{array}{lll}&y_1(x)\sim 1,\ \ \ \ \ &y_1'(x)\sim o(\frac1x),\\
&y_2(x)\sim x,\ \ \ \ \ &y_2'(x)\sim 1.
\end{array}
\]
Here and below, $A\sim B$ means $A\leq C_1B$ and $B\leq C_2A$ for some constants $C_1$ and $C_2$.
Picking $y=y_1$ then we have $k(t)\sim 1$ for $t$ large and
$$\int_{t_0}^\infty r_2(\tau)d\tau=\lim_{t\rightarrow\infty}\int_{t_0}^t-\frac{k'(\tau)}{k(\tau)}d\tau=\lim_{t\rightarrow\infty}\ln(k(t))\left|^t_{t_0}\right.<\infty. $$
Hence, we obtain the $L^1$ integrability of $r_1(t)$ and $r_2(t)$.
Moreover, we can check the sign of $r_2(t)$ if $b(t)>0$. According to \eqref{ricatti}, $$y''=-c^2(t)b(t)y,$$
for sufficiently large $t>t_0$ or equivalently sufficiently large $x>0$, if $y(x)\sim 1$ is positive (\textit{resp.} negative), then $y''<0\ (\textit{resp}. >0)$ which means $y'$ is decreasing (\textit{resp}. increasing) with respect to $x$.
Since $y'\sim o(\frac1x)$ which goes to $0$, $y'$ must be positive (\textit{resp}. negative). From
$$r_2(t)=-\frac{k'(t)}{k(t)}=-\frac1{c(t)}\frac{y'(x)}{y(x)},$$ we know $r_2(t)$ is negative. $ \hfill{} \Box$\\

\begin{remark}\label{rem-App}
In fact, in order to apply the Lemma 2.3 in \cite{WY-damped}, we only need to check the sign of $r_1(t)-r_2(t)$. Setting
 $2v(t):=r_1(t)-r_2(t)=a(t)-2r_2(t)$, one may find $v(t)$ satisfies:
$$v'(t)+v^2(t)=\frac12a'(t)+\frac14a^2(t)-b(t).$$
With some similar argument as above for $r_2(t)$, one can also obtain the $L^1$ integrability of $v$ and hence $r_1$ and $r_2$.
Moreover, if $\frac12a'(t)+\frac14a^2(t)-b(t)<0$ for large $t$, one obtain $v'<0$ which means that $v$ decreases, with the $L^1$ integrability, one may finally reduce the positivity of $v$.
In this sense, we can relax the non-negativity of $b(t)$ to $b(t)>\frac12a'(t)+\frac14a^2(t)$, while in the latter case, $b(t)$ could be negative.
However, the authors believe the essential improvement shall come from the improvement of Lemma 2.3 in \cite{WY-damped}. So far, we do not know how to get ride of the positive requirement of $a(t)$ there.
\end{remark}
\textit{Proof of Lemma \ref{lem:rho}:}
In order to show this claim, one may set $\eta(t)=\rho(t)e^{\lambda t}$ which satisfies
$$\eta''-(2\lambda+a)\eta'+(\lambda a+b-a')\eta=0.$$
Set
$$c(t)=e^{-\int_{t_0}^t2\lambda+a(\tau)d\tau},\ \ x =\int_{t_0}^t\frac 1{c(s)}ds,\ \ y(x)=\eta(t),$$
for some large $t_0$, one find that $y(x)$ satisfies:
\begin{equation}\label{y}
\frac{d^2y}{dx^2}+c^2(t)(\lambda a+ b-a')y(x)=0.
\end{equation}
Since
$$\int_{t_0}^t\frac{c(t)}{c(s)}ds=\int_{t_0}^te^{-\int_{s}^t2\lambda+a(\tau)d\tau}ds\leq e^{\|a\|_{L^1}}\int_{t_0}^te^{-2\lambda(t-s)}ds<\frac12e^{\|a\|_{L^1}},$$
so one obtain
$$\int_{0}^\infty x|c^2(t)(\lambda a+b-a')|dx=\int_{t_0}^\infty|x c(t)(\lambda a+b-a') |dt<\frac12e^{\|a\|_{L^1}}\int_{t_0}^\infty|\lambda a+b-a'|dt$$
which is finite. Hence, with similar argument as \eqref{ricatti}, \eqref{y} has a solution $y(x)\sim 1$ for large $x$, which implies $\eta(t)\sim 1$ and further $\rho(t)\sim e^{-\lambda t}$ for large $t$.

\section*{Acknowledgments} The authors would like to thank the referee for the careful reading and useful suggestions.
The first and third authors were partially supported by Grant-in-Aid for Scientific Research (No.18H01132), JSPS.
The second author was partially supported by Zhejiang Provincial Nature Science Foundation of China under Grant No. LY18A010023.

\end{document}